\documentclass[12pt]{amsart}
\usepackage{amsmath,amssymb}
\usepackage{graphicx,bm}
\usepackage[margin=1in]{geometry}
\usepackage{hyperref}
\usepackage{amsthm}
\usepackage{verbatim}

\numberwithin{equation}{section}

\newtheorem{thm}{Theorem}[section]

\newtheorem{cor}[thm]{Corollary}
\newtheorem{lem}[thm]{Lemma}

\newtheorem{defn}[thm]{Definition}
\newtheorem{rmk}[thm]{Remark}

\begin{document}

\title[Weakly Horospherically Convex Hypersurfaces in Hyperbolic Space]{Weakly Horospherically Convex Hypersurfaces in Hyperbolic Space}
\author{Vincent Bonini, Jie Qing$^1$, Jingyong Zhu$^2$}
\thanks{$^1$The author would like to acknowledge the partial support by NSF DMS-1608782 for this research.}
\thanks{$^2$The author wants to thank the support by China Scholarship Council for visiting University of California, Santa Cruz. }
\address{Department of Mathematics,  Cal Poly State University, San Luis Obispo, CA 93407, United States}
\email{vbonini@calpoly.edu}
\address{Department of Mathematics,  University of California, Santa Cruz, CA 95064, United States}
\email{qing@ucsc.edu}
\address{School of Mathematical Sciences,  University of Science and Technology of China, Hefei, Anhui 230026, China}
\email{zjyjj0224@gmail.com}

\subjclass[2010]{}%
\keywords{weakly horospherically convex, hyperbolic space, support function, embeddedness, Bernstein theorem}


\begin{abstract}

In \cite{bonini2015hypersurfaces}, the authors develop a global correspondence between immersed weakly horospherically convex hypersurfaces $\phi:M^n \to \mathbb{H}^{n+1}$ and a class of conformal metrics on domains of the round sphere $\mathbb{S}^n$. Some of the key aspects of the correspondence and its consequences have dimensional restrictions $n\geq3$ due to the reliance on an analytic proposition from \cite{chang2004class} concerning the asymptotic behavior of conformal factors of conformal metrics on domains of $\mathbb{S}^n$. In this paper, we prove a new lemma about the asymptotic behavior of a functional combining the gradient of the conformal factor and itself, which  allows us to extend the global correspondence and embeddedness theorems of \cite{bonini2015hypersurfaces} to all dimensions $n\geq2$ in a unified way. In the case of a single point boundary $\partial_{\infty}\phi(M)=\{x\} \subset \mathbb{S}^n$, we improve these results in one direction. As an immediate consequence of this improvement and the work on elliptic problems in \cite{bonini2015hypersurfaces}, we have a new, stronger Bernstein type theorem. Moreover, we are able to extend the Liouville and Delaunay type theorems from \cite{bonini2015hypersurfaces} to the case of surfaces in $\mathbb{H}^{3}$.
\end{abstract}
\maketitle


\section{Introduction}
    
In \cite{espinar2009hypersurfaces} the authors observed an interesting fact that the principal curvatures of a hypersurface $\phi : M^n\to\mathbb{H}^{n+1}$ satisfying an exterior horosphere condition, which we call weak horospherical convexity, can be calculated in terms of the eigenvalues of the Schouten tensor of the horospherical metric $\hat{g}=e^{2\rho}g_{\mathbb{S}^n}$ via its horospherical support function $\rho$. This observation creates a local correspondence that opens a window for more interactions between the study of elliptic problems of Weingarten hypersurfaces in hyperbolic space and the study of elliptic problems of conformal metrics on domains of $\mathbb{S}^n$.

In \cite{bonini2015hypersurfaces} the authors developed the global theory for the correspondence above. They established results on when the hyperbolic Gauss map $G: M^n\to\mathbb{S}^n$ is injective and discussed when an immersed weakly horospherically convex hypersurface can be unfolded into an embedded one along the normal flow. As applications, they established a new Bernstein type theorem for a complete, immersed, weakly horospherically convex hypersurfaces in $\mathbb{H}^{n+1}$ of constant mean curvature and new Liouville and Delaunay type theorems. 

In one direction injectivity of the hyperbolic Gauss map $G:M^n \to \mathbb{S}^n$ is found to play an essential role in both the correspondence and the ability to unfold such hypersurfaces into embedded ones. In the other direction two key issues are the properness of a hypersurface associated to a conformal metric and the consistency of its boundary at infinity with that of its hyperbolic Gauss map image. These issues are resolved in \cite{bonini2015hypersurfaces} with the aid of Proposition 8.1 in \cite{chang2004class}, which is used to control the asymptotic behavior of conformal factors of conformal metrics on domains of $\mathbb{S}^n$ and guarantee both the properness of associated hypersurfaces and the consistency of the aforementioned boundaries. For technical reasons Proposition 8.1 in \cite{chang2004class} is not valid when $n=2$, so its use places the dimensional restrictions $n\geq3$ on the global correspondence and its applications.
  
In this paper, we first establish a new lemma on the asymptotic behavior of a functional of the conformal factor for realizable metrics on domains of $\mathbb{S}^n$ (see Definition \eqref{Def:AdmissibleRealizable}).

\begin{lem}\label{Lem:Gradient}
For $n \geq 2$, suppose that $\hat{g}=e^{2\rho}g_{\mathbb{S}^n}$ is a complete conformal metric on a domain $\Omega\subset\mathbb{S}^n$ with bounded $2$-tensor 
$P=-\nabla_{g_{\mathbb{S}^n}}^2\rho+d\rho\otimes d\rho-\frac12(|\nabla^{g_{\mathbb{S}^n}}\rho|_{g_{\mathbb{S}^n}}^2-1)g_{\mathbb{S}^n}$. Then 
\begin{equation}
e^{2\rho(x)}+|\nabla^{g_{\mathbb{S}^n}}\rho|_{g_{\mathbb{S}^n}}^2(x)\to+\infty \  \  \text{as} \  \ x\to x_0\in\partial\Omega.
\end{equation}
\end{lem}

Using Lemma \ref{Lem:Gradient} in place of Proposition 8.1 \cite{chang2004class}, we address the questions in \cite{bonini2015hypersurfaces} for all dimensions $n\geq2$ in a unified way. In particular, we can extend two of the main results in \cite{bonini2015hypersurfaces} to include the lower dimensional case $n=2$. The first is the Global Correspondence Theorem between admissible hypersurfaces and realizable metrics (see Definition \eqref{Def:AdmissibleRealizable}).
 
\begin{thm}\label{Thm:GCT}[Global Correspondence Theorem]
For $n \geq 2$, suppose that $\phi : M^n\to\mathbb{H}^{n+1}$ is an immersed, complete, uniformly weakly horospherically convex hypersurface with injective hyperbolic Gauss map $G : M^n\to\mathbb{S}^{n} $. Then it induces a complete conformal metric $\hat{g}=e^{2\rho}g_{\mathbb{S}^n}$ on the image of the hyperbolic Gauss map $G(M)\subset\mathbb{S}^{n}$ with bounded $2$-tensor 
\begin{equation}\label{Eq:2TensorP}
P=-\nabla_{g_{\mathbb{S}^n}}^2\rho+d\rho\otimes d\rho-\frac12(|\nabla^{g_{\mathbb{S}^n}}\rho|_{g_{\mathbb{S}^n}}^2-1)g_{\mathbb{S}^n},
\end{equation} 
where $\rho$ is the horospherical support function of $\phi$ and
\begin{equation}\label{Eq:Boundary}
\partial_{\infty}\phi(M)=\partial G(M).
\end{equation}

On the other hand, suppose that $\hat{g}=e^{2\rho}g_{\mathbb{S}^{n}}$ is a complete conformal metric on a domain $\Omega\subset\mathbb{S}^{n}$ with bounded $2$-tensor $P$ defined as in \eqref{Eq:2TensorP}. Then it induces a properly immersed, complete, uniformly weakly horospherically convex hypersurface
\begin{equation}
\phi^t=\frac{e^{\rho+t}}{2}(1+e^{-2(\rho+t)}(1+|\nabla^{g_{\mathbb{S}^n}}\rho|_{g_{\mathbb{S}^n}}^2))(1,x)+e^{-(\rho+t)}(0,-x+\nabla^{g_{\mathbb{S}^n}}\rho) : \Omega\to\mathbb{H}^{n+1}
\end{equation}
for $t$ sufficiently large.
\end{thm}
  
The other main result is the Embeddedness Theorem, which shows that a uniformly weakly horospherically convex hypersurface will eventually become embedded under the geodesic normal flow when the hypersurface has boundary at infinity with relatively simple structure.

\begin{thm}\label{Thm:EmbFlow}[Embeddedness Theorem]
For $n \geq 2$, suppose that $\phi : M^n\to\mathbb{H}^{n+1}$ is an immersed, complete, uniformly weakly horospherically convex hypersurface with injective hyperbolic Gauss map $G : M^n\to\mathbb{S}^{n}$. In addition, assume that the boundary at infinity $\partial_{\infty}\phi(M)$ is a disjoint, finite union of smooth compact embedded submanifolds with no boundary in $\mathbb{S}^{n}$. Then $\phi$ can be unfolded into an embedded hypersurface along its geodesic normal flow eventually.

Equivalently, suppose that $\hat{g}=e^{2\rho}g_{\mathbb{S}^{n}}$ is a complete conformal metric on a domain $\Omega\subset\mathbb{S}^{n}$ with bounded $2$-tensor $P$ defined as in \eqref{Eq:2TensorP}. In addition, assume that the boundary $\partial\Omega$ is a disjoint, finite union of smooth compact embedded submanifolds with no boundary in $\mathbb{S}^{n}$. Then the hypersurfaces
\begin{equation}
\phi^t=\frac{e^{\rho+t}}{2}(1+e^{-2(\rho+t)}(1+|\nabla^{g_{\mathbb{S}^n}}\rho|_{g_{\mathbb{S}^n}}^2))(1,x)+e^{-(\rho+t)}(0,-x+\nabla^{g_{\mathbb{S}^n}}\rho) : \Omega\to\mathbb{H}^{n+1}
\end{equation}
are embedded for $t$ sufficiently large. 
\end{thm}
 
Due to Corollary 4.4 and Proposition 1.3 in \cite{epstein1987asymptotic}, when the asymptotic boundary of a hypersurface $\phi:M^n \to \mathbb{H}^n$ is a single point $\partial_{\infty}\phi(M)=\{x\} \subset \mathbb{S}^n$, we are able to recover one direction of the Global Correspondence Theorem \ref{Thm:GCT} and the Embeddedness Theorem \ref{Thm:EmbFlow} without a prior assuming that the hyperbolic Gauss map $G:M^n \to \mathbb{S}^n$ is injective.

\begin{thm}\label{Thm:CorrSinglePtBd}
For $n\geq 2$, suppose that $\phi : M^n\to\mathbb{H}^{n+1}$ is an immersed, complete, uniformly weakly horospherically convex hypersurface with the boundary at infinity $\partial_{\infty}\phi(M)=\{x\} \subset \mathbb{S}^n$ a single point. Then it induces a complete conformal metric $\hat{g}=e^{2\rho}g_{\mathbb{S}^{n}}$ on the image of the hyperbolic Gauss map $G(M)\subset\mathbb{S}^{n}$ with bounded $2$-tensor $P$ defined as in \eqref{Eq:2TensorP} where $\rho$ is the horospherical support function of $\phi$ and
\begin{equation}\label{Eq:Boundary}
\partial_{\infty}\phi(M)=\partial G(M).
\end{equation}
Moreover, $\phi$ can be unfolded into an embedded hypersurface along its geodesic normal flow eventually.
\end{thm}

Together with the work on elliptic problems in \cite{bonini2015hypersurfaces}, Theorem \ref{Thm:CorrSinglePtBd} leads to a stronger Bernstein type theorem for hypersurfaces in $\mathbb{H}^{n+1}$ .

\begin{thm}\label{Thm:Bernstein}
For $n\geq2$, suppose that $\phi: M^n\to\mathbb{H}^{n+1}$ is an immersed, complete, uniformly weakly horospherically convex hypersurface with constant mean curvature. Then it is a horosphere if its boundary at infinity is a single point in $\mathbb{S}^n$.
\end{thm}

Moreover, with the global correspondence for dimension $n=2$ in hand, we can extend the Liouville and Delaunay type theorems in \cite{bonini2015hypersurfaces} to the cases of conformal metrics on domains on the standard $2$-sphere $\mathbb{S}^2$ and surfaces immersed in hyperbolic space ${\mathbb H}^3$. Collectively, these results demonstrate the usefulness of the global correspondence by providing a bridge between elliptic problems of hypersurfaces and those of conformal metrics. In particular, under the correspondence the Bernstein theorem for surfaces is found to be equivalent to the following Liouville type result for conformal metrics, whose higher dimensional version can be found in \cite{caffarelli1989asymptotic}.

\begin{cor}\label{Cor:Liouville}
Let $g=e^{2\rho}g_{\mathbb{S}^2}$ be a complete, conformal metric on ${\mathbb S}^2\setminus\{p\}$ with bounded $2$-tensor $P_g$. If the eigenvalues of $P_g$ satisfy a conformally invariant elliptic problem of conformal metrics as defined in \cite{bonini2015hypersurfaces}, then $g$ must be the Euclidean metric.
\end{cor}

At last, we remove the restriction $n\geq3$ on the dimension of the hypersurfaces in the Delaunay Theorem in \cite{bonini2015hypersurfaces}.

\begin{cor}\label{Cor:Delaunay}
Let $\phi: M^2\to\mathbb{H}^{3}$ be an immersed, complete, uniformly weakly horospherically convex hypersurface with boundary at infinity $\partial_{\infty}\phi(M)=\{p,q\}$ consisting of exactly two points. If the principal curvatures of $\phi$ satisfy an elliptic Weingarten equation as defined in \cite{bonini2015hypersurfaces}, then $\phi$ is rotationally symmetric with respect to the geodesic joining the two points at the infinity of $\phi$.  
Equivalently, let $g$ be a complete, conformal metric satisfying a conformally invariant elliptic problem of conformal metrics on $\Omega=\mathbb{S}^2\setminus\{p, q\}$. Then $g$ is cylindric with respect to the geodesic joining the two points $p$ and $q$.
\end{cor}

\noindent
Corollaries \ref{Cor:Liouville} and \ref{Cor:Delaunay}  are immediate consequences of the global correspondence in two dimensions and the work on elliptic problems in \cite{bonini2015hypersurfaces}. Since our purpose is mainly to demonstrate the usefulness of the global correspondence between hypersurfaces and conformal metrics, we refer the reader to \cite{bonini2015hypersurfaces} for their precise statements and proofs.
 
This article is organized as follows. In section \ref{Sect:Local}, we will recall the local theory developed in \cite{espinar2009hypersurfaces}. In section \ref{Sect:Global}, we establish Lemma \ref{Lem:Gradient} and consequently the Global Correspondence Theorem \ref{Thm:GCT} for all dimensions $n\geq2$. In section \ref{Sect:FlowEmb} we use the global correspondence to establish the Embeddness Theorem \ref{Thm:EmbFlow}. In one direction, we are able to improve both of these results in the special case of hypersurfaces with single point asymptotic boundaries. Based on the techniques of \cite{bonini2015hypersurfaces}, we then use this improved correspondence to prove our new, stonger Bernstein type theorem.


\section{Local Theory}\label{Sect:Local}

In this section we introduce the basic constructions and terminology used in \cite{bonini2015hypersurfaces},\cite{bonini2016nonnegatively},\cite{espinar2009hypersurfaces} and references therein. We also restate the local correspondence developed in \cite{espinar2009hypersurfaces} between hypersurfaces in hyperbolic space $\mathbb{H}^{n+1}$ and conformal metrics on domains of $\mathbb{S}^n$. 
 

\subsection{Weak Horospherical Convexity and the Horospherical Metric}  

For $n\geq2$, let us denote Minkowski spacetime by $\mathbb{R}^{1,n+1}$, that is, the vector space $\mathbb{R}^{n+2}$ endowed with the Minkowski spacetime metric $\langle, \rangle$ given by\begin{equation*}
\langle \bar{x},\bar{x}\rangle=-x_0^2+\sum_{i=1}^{n+1}x_i^2,
\end{equation*}
where $\bar{x}=(x_0,x_1,\dots,x_{n+1})\in\mathbb{R}^{n+2}$. Then hyperbolic space, de Sitter spacetime and the positive null cone are given by
\begin{equation*}
\begin{split}
& \mathbb{H}^{n+1}=\{\bar{x}\in\mathbb{R}^{1,n+1}|\langle \bar{x},\bar{x}\rangle=-1, x_0>0\},\\
&\mathbb{S}^{1,n}=\{\bar{x}\in\mathbb{R}^{1,n+1}|\langle \bar{x},\bar{x}\rangle=1\},\\
&\mathbb{N}^{n+1}_+=\{\bar{x}\in\mathbb{R}^{1,n+1}|\langle \bar{x},\bar{x}\rangle=0, x_0>0\},
\end{split}  
\end{equation*}
respectively. We identify the ideal boundary at infinity of hyperbolic space $\mathbb{H}^{n+1}$ with the unit round sphere $\mathbb{S}^n$ sitting at height $x_0=1$ in the null cone $\mathbb{N}^{n+1}_+$ of Minkowski space $\mathbb{R}^{1,n+1}$.
 
An immersed hypersurface in $\mathbb{H}^{n+1}$ is given by a parametrization
\begin{equation*}
\phi : M^n\to\mathbb{H}^{n+1}
\end{equation*}
and an orientation assigns a unit normal vector field
\begin{equation*}
\eta : M^n\to\mathbb{S}^{1,n}.
\end{equation*}
Horospheres are used to define the hyperbolic Gauss map of an oriented, immersed hypersurface in $\mathbb{H}^{n+1}$. In the hyperboloid model $\mathbb{H}^{n+1}$, horospheres are the intersections of affine null hyperplanes of $\mathbb{R}^{1,n+1}$ with $\mathbb{H}^{n+1}$.

\begin{defn}\label{Def:GaussMap}
Let $\phi : M^n\to\mathbb{H}^{n+1}$ be an immersed, oriented hypersurface in $\mathbb{H}^{n+1}$ with unit normal field $\eta : M^n\to\mathbb{S}^{1,n}$. The hyperbolic Gauss map 
\begin{equation*}
G : M^n\to\mathbb{S}^{n}
\end{equation*}
of $\phi$ is defined as follows: for every $p\in M^n$, $G(p)\in\mathbb{S}^n$ is the point at infnity of the unique horosphere in $\mathbb{H}^{n+1}$ passing through $\phi(p)$ whose outward unit normal agrees with $\eta(p)$ at $\phi(p)$.
\end{defn}

Given an immersed, oriented hypersurface $\phi : M^n\to\mathbb{H}^{n+1}$ with unit normal field $\eta : M^n\to\mathbb{S}^{1,n}$, the light cone map $\psi$ associated to $\phi$ is defined
\begin{equation*}
\psi=\phi-\eta : M^n\to\mathbb{N}^{n+1}_+.
\end{equation*}  
With the identification of the ideal boundary at infinity of hyperbolic space $\mathbb{H}^{n+1}$ with the unit round sphere $\mathbb{S}^n$ sitting $x_0=1$, we have
\begin{equation*}
\psi=e^{\tilde{\rho}}(1,G),
\end{equation*}  
where $\psi_0=e^{\tilde{\rho}}$ is the horospherical support function of $\phi$. Note that with our convention given in Definition \ref{Def:GaussMap}, horospheres with outward orientation are the unique surfaces such that both the hyperbolic Gauss map and the associated light cone map are constant. Moreover, if $x \in \mathbb{S}^n$ is the point at infinity of such a horosphere, then $\psi = e^{\tilde{\rho}}(1,x)$ where $\tilde{\rho}$ is the signed hyperbolic distance of the horosphere to the point $\mathcal{O}=(1,0,\dots,0) \in \mathbb{H}^{n+1} \subseteq \mathbb{R}^{1,n+1}$. 

We use supporting horospheres to introduce the following notion of weak horospherical convexity.

\begin{defn}\cite{bonini2015hypersurfaces}\cite{espinar2009hypersurfaces}\label{def2.2}
Let $\phi : M^n\to\mathbb{H}^{n+1}$ be an immersed, oriented hypersurface in $\mathbb{H}^{n+1}$ with unit normal field $\eta : M^n\to\mathbb{S}^{1,n}$. Let $\mathcal{H}_p$ denote the horosphere in $\mathbb{H}^{n+1}$ that is tangent to the hypersurface at $\phi(p)$ and whose outward unit normal at $\phi(p)$ agrees with the unit normal $\eta(p)$. We will say that $\phi : M^n\to\mathbb{H}^{n+1}$ is weakly horospherically convex at $p$ if there exists a neighborhood $V\subset M^n$ of $p$ so that $\phi(V\setminus\{p\})$ does not intersect with $\mathcal{H}_p$. Moreover, the distance function of the hypersurface $\phi : M^n\to\mathbb{H}^{n+1}$ to the horosphere $\mathcal{H}_p$ does not vanish up to the second order at $\phi(p)$ in any direction.
\end{defn}
  
From the definition above, we have following corollary.
\begin{cor}\cite{espinar2009hypersurfaces}
Let $\phi : M^n\to\mathbb{H}^{n+1}$ be an immersed, oriented hypersurface in $\mathbb{H}^{n+1}$. Then $\phi$ is weakly horospherically convex at $p$ if and only if all the principal curvatures of $\phi$ at $\phi(p)$ are simultaneously $<-1$ or $>-1$.
\end{cor}
 
Let $\{e_1,\dots,e_n\}$ denote an orthonormal basis of principal directions of $\phi$ at $p$ and let $\kappa_1,\dots,\kappa_n$ denote the associated principal curvatures. Then, as in \cite{espinar2009hypersurfaces}, it follows that
\begin{equation}\label{Eq:LCmetric}
\langle(d\psi)_p(e_i),(d\psi)_p(e_j)\rangle=(1+\kappa_i)^2\delta_{ij}=e^{2\tilde{\rho}}\langle(dG)_p(e_i),(dG)_p(e_j)\rangle_{g_{\mathbb{S}^n}},
\end{equation}
where $\tilde{\rho}$  is the horospherical support function of $\phi$ and $g_{\mathbb{S}^n}$ denotes the standard round metric on $\mathbb{S}^n$. Hence, the hyperbolic Gauss map of a weakly horospherically convex hypersurface is a local diffeomorphism and can be used to define the so-called horospherical metric as follows.

\begin{defn}\label{Def:HorMetric}\cite{bonini2015hypersurfaces}\cite{espinar2009hypersurfaces}
Let $\phi : M^n\to\mathbb{H}^{n+1}$ be an immersed, weakly horospherically convex hypersurface in $\mathbb{H}^{n+1}$. Then the locally conformally flat metric
\begin{equation}
g_h=\psi^*\langle , \rangle=e^{2\tilde{\rho}}G^*g_{\mathbb{S}^n}
\end{equation}
on $M^n$ is called the horospherical metric of $\phi$.
\end{defn}

When the hyperbolic Gauss map $G : M^n\to\mathbb{S}^{n}$ of a weakly horospherically convex hypersurface $\phi : M^n\to\mathbb{H}^{n+1}$ is injective, one can push the horospherical metric $g_h$ onto the image $\Omega=G(M)\subset\mathbb{S}^n$ and consider the conformal metric 
\begin{equation*}
\hat{g}=(G^{-1})^*g_h=e^{2\rho}g_{\mathbb{S}^n},
\end{equation*}
where $\rho=\tilde{\rho}\circ G^{-1}$. For simplicity, we also refer to this conformal metric $\hat{g}$ as the horospherical metric. On the other hand, given a conformal metric $\hat{g}=e^{2\rho}g_{\mathbb{S}^n}$ on a domain $\Omega$ in $\mathbb{S}^n$, one recovers the light cone map 
\begin{equation*}
\psi(x)=e^{\rho}(1,x) : \Omega\to\mathbb{N}^{n+1}_+.
\end{equation*}
Then one can solve for the map $\phi : \Omega\to\mathbb{H}^{n+1}$ and the unit normal vector $\eta : \Omega\to\mathbb{S}^{1,n}$ so that $\psi=\phi-\eta$. These facts and the discussion above lead to the local correspondence developed in  \cite{espinar2009hypersurfaces}.

\begin{thm}[Local Correspondence Theorem \cite{espinar2009hypersurfaces}]\label{Thm:LCT}
For $n\geq 2$, let $\phi : \Omega\subseteq\mathbb{S}^n\to\mathbb{H}^{n+1}$ be a weakly horospherically convex hypersurface with hyperbolic Gauss map $G(x)=x$ the identity. Then $\psi=e^{\rho}(1,x)$ and 
\begin{equation}\label{Eq:Parametrization}
\phi=\frac{e^{\rho}}{2}(1+e^{-2\rho}(1+|\nabla^{g_{\mathbb{S}^n}}\rho|_{g_{\mathbb{S}^n}}^2))(1,x)+e^{-\rho}(0,-x+\nabla^{g_{\mathbb{S}^n}}\rho).
\end{equation}
Moreover, there is a symmetric $2$-tensor 
\begin{equation}\label{Eq:2TensorP2}
P=-\nabla_{g_{\mathbb{S}^n}}^2\rho+d\rho\otimes d\rho-\frac12(|\nabla^{g_{\mathbb{S}^n}}\rho|_{g_{\mathbb{S}^n}}^2-1)g_{\mathbb{S}^n}
\end{equation}
associated to the horospherical metric $\hat{g}=e^{2\rho}g_{\mathbb{S}^n}$ whose eigenvalues $\lambda_i$ are related to the principal curvature $\kappa_i$ of $\phi$ by 
\begin{equation}\label{Eq:LambdaKappa}
\lambda_i=\frac12-\frac{1}{1+\kappa_i}.
\end{equation}
Conversely, suppose that $\hat{g}=e^{2\rho}g_{\mathbb{S}^n}$ is a conformal metric  defined on a domain $\Omega\subseteq\mathbb{S}^n$ such that all the eigenvalues of the $2$-tensor $P$ defined as in \eqref{Eq:2TensorP2} are less than $\frac{1}{2}$. Then the map $\phi : \Omega\to\mathbb{H}^{n+1}$ given by \eqref{Eq:Parametrization} defines an immersed weakly horospherically convex hypersurface with hyperbolic Gauss map $G(x)=x$ the identity on $\Omega$. Moreover, the horospherical metric of $\phi$ is $\hat{g}$ and its principal curvatures satisfy the relation \eqref{Eq:LambdaKappa}.
\end{thm}

\begin{rmk}
When $n\geq3$, the symmetric $2$-tensor $P$ given by \eqref{Eq:2TensorP2} is exactly the Schouten tensor of the conformal metric $\hat{g}=e^{2\rho}g_{\mathbb{S}^n}$.
\end{rmk}

Now for a conformal metric $\hat{g}=e^{2\rho}g_{\mathbb{S}^n}$ on a domain $\Omega\subset\mathbb{S}^n$ with $2$-tensor $P$ bounded from above, one can consider a family of rescaled metric $\hat{g}^t=e^{2t}\hat{g}$. Choosing $t_0$ sufficiently large so that $e^{-2t}\lambda_i<\frac12$ for $t \geq t_0$, it follows from Theorem \ref{Thm:LCT} that the family of hypersurfaces 
\begin{equation}\label{Eq:rescale}
\phi^t=\frac{e^{\rho+t}}{2}(1+e^{-2(\rho+t)}(1+|\nabla^{g_{\mathbb{S}^n}}\rho|_{g_{\mathbb{S}^n}}^2))(1,x)+e^{-(\rho+t)}(0,-x+\nabla^{g_{\mathbb{S}^n}}\rho) : \Omega\to\mathbb{H}^{n+1}
\end{equation}
are immersed and weakly horospherically convex with hyperbolic Gauss maps the identity for $t\geq t_0$. Moreover, the eigenvalues $\lambda_i^t=e^{-2t}\lambda_i$ of the $2$-tensor $P_t$ associated to $\hat{g}_t$ and the principal curvatures $\kappa_i^t$ of the associated hypersurfaces $\phi_t$ satisfy the relation
\begin{equation}
\lambda_i^t=\frac{1}{2}-\frac{1}{1+\kappa_i^t}.
\end{equation}


\section{Global Theory}\label{Sect:Global}

In this section, we establish a global correspondence between properly immersed, complete, weakly horospherically convex hypersurfaces and complete conformal metrics on the domains of $\mathbb{S}^n$ for all dimensions $n\geq2$. To ensure a two-sided correspondence, we restrict ourselves to the cases of uniformly weakly horospherically convex hypersurfaces and conformal metrics with bounded 2-tensor $P$. 

\begin{defn}\cite{bonini2015hypersurfaces}
For $n\geq 2$, let $\phi: M^n\to\mathbb{H}^{n+1}$ be an immersed, oriented hypersurface. We say that $\phi$ is uniformly weakly horospherically convex if there is a constant $\kappa_0>-1$ such that the principal curvatures $\kappa_1, \dots, \kappa_n$ are greater than or equal to $\kappa_0$ at all points in $M^n$.
\end{defn}

One reason why we are interested in uniformly weakly horospherically convex hypersurfaces is the fact that under such curvature assumptions, the completeness of such hypersurfaces is equivalent to that of corresponding conformal metrics on domains of $\mathbb{S}^n$ due to \eqref{Eq:LCmetric}. It turns out that completeness is very important in establishing the consistency of boundaries. On the other hand, from the curvature relation \eqref{Eq:LambdaKappa}, one can easily see that given a conformal metric on a domain of $\mathbb{S}^n$ with 2-tensor $P$ bounded above, the corresponding immersed hypersurfaces $\phi^t$ given by \eqref{Eq:rescale} with $t$ sufficiently large are uniformly weakly horospherically convex if and only if $P$ is also bounded below. Based on these observations we focus on the classes of hypersurfaces and conformal metrics in the following definition.

\begin{defn}\label{Def:AdmissibleRealizable}
For $n\geq2$, an oriented hypersurface $\phi: M^n\to\mathbb{H}^{n+1}$ is said to be $\mathbf{admissible}$ if it is properly immersed, complete, and uniformly weakly horospherically convex with injective hyperbolic Gauss map $G: M^n\to\mathbb{S}^n$. Meanwhile, a complete metric $\hat{g}=e^{2\rho}g_{\mathbb{S}^n}$ on a domain $\Omega\subset\mathbb{S}^n$ is called a $\mathbf{realizable}$ metric if its $2$-tensor $P$ defined as in \eqref{Eq:2TensorP2} is bounded.
\end{defn}

Now, we can state our Global Correspondence Theorem.

\begin{thm}[Global Correspondence Theorem]\label{Thm:GCT2}
For $n\geq2$, an admissible hypersurface $\phi : M^n\to\mathbb{H}^{n+1}$ induces a realizable metric on the hyperbolic Gauss map image $\Omega=G(M)\subset\mathbb{S}^n$ with boundary at infinity $\partial_\infty\phi(M)=\partial\Omega$.

On the other hand, given a realizable metric $\hat{g}=e^{2\rho}g_{\mathbb{S}^n}$ on a domain $\Omega\subset\mathbb{S}^n$, the map $\phi^t$ given by \eqref{Eq:rescale} defines an admissible surface with $\partial_\infty\phi^t(\Omega)=\partial\Omega$ for all $t$ sufficiently large.
\end{thm}
 
For higher dimension $n\geq3$, Lemma 3.2 and Corollary 3.1 in \cite{bonini2015hypersurfaces} are used to obtain the properness of the immersed hypersurfaces associated to a realizable metric as well as the consistency of the boundary at the infinity of the hypersurfaces with that of their hyperbolic Gauss map images. These results are independent of the specific dimension and can be stated together as follows.

\begin{lem}\label{Lem:CompleteProperBd}\cite{bonini2015hypersurfaces}
For $n\geq2$, suppose that $\hat{g}=e^{2\rho}g_{\mathbb{S}^n}$ is a complete conformal metric on a domain $\Omega\subset\mathbb{S}^n$ with $2$-tensor $P$ bounded from above. If 
\begin{equation}\label{eq3.1}
\beta(x):=e^{2\rho(x)}+|\nabla^{g_{\mathbb{S}^n}}\rho|_{g_{\mathbb{S}^n}}^2(x)\to+\infty \  \  \text{as} \  \ x\to\partial\Omega,
\end{equation}
then $\phi^t : \Omega\to\mathbb{H}^{n+1}$ given by \eqref{Eq:rescale} is a properly immersed, complete, weakly horospherically convex surface with $\partial_\infty\phi^t(\Omega)=\partial\Omega$ for all $t$ sufficiently large.
\end{lem}

Due to Lemma \ref{Lem:CompleteProperBd} the only issue left in completing the global correspondence is to determine when $\beta(x)$ goes to $+\infty$ as $x$ approaches the boundary $\partial\Omega$. In \cite{bonini2015hypersurfaces}, the authors used Proposition 8.1 in \cite{chang2004class}, which showed that the conformal factor $\rho(x)$ goes to $+\infty$ as $x$ approaches the boundary $\partial\Omega$ when the scalar curvature of a complete conformal  metric $\hat{g}=e^{2\rho}g_{\mathbb{S}^n}$ is bounded from below. The proof of Proposition 8.1 in \cite{chang2004class} relies on the Moser iteration and fails to hold in dimension $n=2$. 

In the remainder of this section, we establish a new key lemma about the asymptotic behavior of $\beta(x)$ in \eqref{eq3.1} from which the Global Correspondence Theorem \ref{Thm:GCT2} follows immediately from Lemma \ref{Lem:CompleteProperBd}. We would like to point out that unlike Proposition 8.1 in \cite{chang2004class}, the following lemma relies on both the upper and lower bounds of the associated conformal $2$-tensor $P$ but its use allows us to treat the global correspondence in all dimensions $n\geq2$ in a unified way. 

\begin{lem}\label{Lem:Gradient2}
For $n\geq 2$, suppose that $\hat{g}=e^{2\rho}g_{\mathbb{S}^n}$ is a realizable metric on a domain $\Omega\subset\mathbb{S}^n$. Then 
\begin{equation}
e^{2\rho(x)}+|\nabla^{g_{\mathbb{S}^n}}\rho|_{g_{\mathbb{S}^n}}^2(x)\to+\infty \  \  \text{as} \  \ x\to x_0\in\partial\Omega.
\end{equation}
\end{lem}

\begin{proof}
For sake of contradiction, assume to the contrary that $\{x_i\} \subset \Omega$ is a sequence of interior points such that
\begin{equation}
x_i\to x_0\in\partial\Omega \  \  \text{as} \  \  i\to+\infty
\end{equation}
and
\begin{equation}\label{Eq:BetaBd}
\beta(x_i) = e^{2\rho(x_i)}+|\nabla^{g_{\mathbb{S}^n}}\rho|_{g_{\mathbb{S}^n}}^2(x_i)\leq C, \  \forall i,
\end{equation}
for some positive constant $C$. Since the $2$-tensor $P=-\nabla_{g_{\mathbb{S}^n}}^2\rho+d\rho\otimes d\rho-\frac12(|\nabla^{g_{\mathbb{S}^n}}\rho|_{g_{\mathbb{S}^n}}^2-1)g_{\mathbb{S}^n}$ is bounded, we find
 \begin{equation}
 |\nabla_{g_{\mathbb{S}^n}}^2\rho|_{g_{\mathbb{S}^n}}\leq K(C_0e^{2\rho}+|\nabla^{g_{\mathbb{S}^n}}\rho|_{g_{\mathbb{S}^n}}^2(x)+1),
 \end{equation}
where $|P|_{\hat g} \leq C_0$ and $K=\max\{1, \frac{\sqrt{n}}{2}\}$. Then by Kato's inequality we have
\begin{equation}\label{Eq:EvolutionIneq}
\frac{\partial}{\partial t}|\nabla^{g_{\mathbb{S}^n}}\rho|_{g_{\mathbb{S}^n}}(x)\leq|\frac{\partial}{\partial t}|\nabla^{g_{\mathbb{S}^n}}\rho|_{g_{\mathbb{S}^n}}|(x)\leq  
|\nabla_{g_{\mathbb{S}^n}}^2\rho|_{g_{\mathbb{S}^n}}(x)\leq K(C_0e^{2\rho}+|\nabla^{g_{\mathbb{S}^n}}\rho|_{g_{\mathbb{S}^n}}^2+1),
\end{equation}
where $\frac{\partial}{\partial t}$ denotes an arbitrary unit tangent vector at $x\in\mathbb{S}^n$.

Now consider $|\nabla^{g_{\mathbb{S}^n}}\rho|_{g_{\mathbb{S}^n}}$ as a function of $t$ along the geodesic starting from $x_i$ in the direction of $\frac{\partial}{\partial t}$. In order to contradict the completeness of the conformal metric $\hat{g}=e^{2\rho}g_{\mathbb{S}^n}$, we aim to derive a uniform bound for $\rho$ in a neighborhood of $x_0$ near the infinity. To do so, we begin by deriving a uniform estimate for $|\nabla^{g_{\mathbb{S}^n}}\rho|_{g_{\mathbb{S}^n}}$ in geodesic balls about $x_i$ with a uniform radius independent of $x_i$ by comparison of the evolution of $|\nabla^{g_{\mathbb{S}^n}}\rho|_{g_{\mathbb{S}^n}}$ with following ODE along geodesics:
\begin{equation}\label{Eq:ODE}
\begin{cases}
Y'=\frac{dY}{dt}=Y^2+A\\
Y(0)=|\nabla^{g_{\mathbb{S}^n}}\rho|_{g_{\mathbb{S}^n}}(x_i)=y_i\leq C,
\end{cases}
\end{equation}
where $A$ is a positive constant to be determined. The solution to \eqref{Eq:ODE} is 
\begin{equation}
\begin{split}
Y&=\sqrt{A}\tan(\sqrt{A}t+\arctan(\frac{1}{\sqrt{A}}y_i))\\
&\leq\sqrt{A}\tan(\sqrt{A}t+\arctan(\frac{1}{\sqrt{A}}C)).
\end{split}
\end{equation} 

Obviously, $\arctan(\frac{1}{\sqrt{A}}C)<\frac\pi2$. Let $\delta=\frac{1}{2\sqrt{A}}(\frac\pi2-\arctan(\frac{1}{\sqrt{A}}C)) > 0$. Then $Y$ is strictly increasing on $[0,\delta)$ and 
\begin{equation}
Y(t)\leq\sqrt{A}\tan(\frac\pi4+\frac12\arctan(\frac{1}{\sqrt{A}}C))=:\bar{Y}, \ \forall t\in[0,\delta).
\end{equation}
Let $A$ be a positive constant such that 
\begin{equation}
2K\delta\bar{Y}=K(\frac\pi2-\arctan(\frac{1}{\sqrt{A}}C)\tan(\frac\pi4+\frac12\arctan(\frac{1}{\sqrt{A}}C))\leq\ln(\frac{A-1}{C_0C}).
\end{equation}
For such fixed $A$, we have $\delta$ and $\bar{Y}$ depending only on $C,C_0$ with
\begin{equation}
C_0Ce^{2K\delta\bar{Y}}+1 \leq A.
\end{equation}

We claim that 
\begin{equation}\label{Eq:UnifGradBd}
|\nabla^{g_{\mathbb{S}^n}}\rho|_{g_{\mathbb{S}^n}}(x) \leq K\bar{Y}, \ \forall x \in B_{\delta}(x_i) \cap \Omega, \ \forall i,
\end{equation}
which is the desired uniform bound on the gradient of the conformal factor in geodesic balls about $x_i$ with uniform radius $\delta$ independent of $x_i$. Otherwise, there exists $0< \tilde{\delta} < \delta$ and some $x_i$ such that
\begin{equation}\label{Eq:IntMax}
\max_{B_{\tilde{\delta}}(x_i)} |\nabla^{g_{\mathbb{S}^n}}\rho|_{g_{\mathbb{S}^n}} =K\bar{Y}.
\end{equation}
But then by the mean value theorem it follows
\begin{equation}
C_0 e^{2\rho(x)} +1 \leq C_0 e^{2(\rho(x_i) + \delta K\bar{Y})} +1\leq C_0C e^{2K\delta \bar{Y}} +1\leq A, \ \forall x \in B_{\tilde{\delta}}(x_i) \cap \Omega,
\end{equation}
and therefore from inequality \eqref{Eq:EvolutionIneq} we find
\begin{equation}\label{Eq:EvolutionIneq2}
\frac{\partial}{\partial t}|\nabla^{g_{\mathbb{S}^n}}\rho|_{g_{\mathbb{S}^n}}(x) \leq K(C_0e^{2\rho(x)}+|\nabla^{g_{\mathbb{S}^n}}\rho|_{g_{\mathbb{S}^n}}^2(x)+1) \leq K(|\nabla^{g_{\mathbb{S}^n}}\rho|_{g_{\mathbb{S}^n}}^2(x) + A).
\end{equation}
Given $x \in B_{\tilde{\delta}}(x_i) \cap \Omega$, let $\alpha$ denote the unit speed geodesic from $x_i$ to $x$ with $x=\alpha(t)$ for some $t \in [0,\tilde{\delta}]$. Comparing with the solution $Y$ to \eqref{Eq:ODE}, it follows from \eqref{Eq:EvolutionIneq2} that
\begin{equation}
|\nabla^{g_{\mathbb{S}^n}}\rho|_{g_{\mathbb{S}^n}}(x) \leq KY(t).
\end{equation}
But $Y$ is strictly increasing on $[0, \delta)$ and $0 < \tilde{\delta} < \delta$ so
\begin{equation}
|\nabla^{g_{\mathbb{S}^n}}\rho|_{g_{\mathbb{S}^n}}(x) < K\bar{Y}, \ \forall x \in  B_{\tilde{\delta}}(x_i) \cap \Omega,
\end{equation} 
contradicting the assumption \eqref{Eq:IntMax} that the maximum value $K\bar{Y}$ was achieved in $B_{\tilde{\delta}}(x_i) \cap \Omega$. Hence, there is no such $\tilde{\delta}$ and our uniform bound \eqref{Eq:UnifGradBd} follows.
 
Now fix $0< \delta' < \frac{1}{2} \delta$. By assumption $x_i\to x_0\in\partial\Omega$ so there is an $i_0$ such that $x_i\in B_{\delta'}(x_0)\cap\Omega$ for all $i\geq i_0$. Then for any $x\in B_{\delta'}(x_0)\cap\Omega$, we have
\begin{equation}\label{Eq:TriangleIneq}
{\rm dist}_{\mathbb{S}^n}(x,x_i)\leq {\rm dist}_{\mathbb{S}^n}(x,x_0)+ {\rm dist}_{\mathbb{S}^n}(x_i,x_0)\leq 2\delta' < \delta, \ \forall i \geq i_0,
\end{equation}
so that 
\begin{equation}
x\in B_{\delta}(x_i)\cap\Omega, \ \forall  i\geq i_0.
\end{equation}
Together with \eqref{Eq:UnifGradBd}, we have
\begin{equation}
|\nabla^{g_{\mathbb{S}^n}}\rho|_{g_{\mathbb{S}^n}}(x)\leq K\bar{Y}, \ \forall x\in B_{\delta'}(x_0)\cap\Omega,
\end{equation}
and therefore
\begin{equation}\label{Eq:MVT}
|\rho(x)|\leq|\rho(x_i)|+K\delta\bar{Y}, \ \forall x\in B_{\delta'}(x_0)\cap\Omega, \ \forall i\geq i_0.
\end{equation}
by the mean value theorem.
For any $j\geq i_0$, $x=x_j$ satisfies \eqref{Eq:MVT} so $|\rho(x_i)|$ are uniformly bounded for $i\geq i_0$. Then we take the infimum to find
\begin{equation}
|\rho(x)|\leq\inf_{i\geq i_0}|\rho(x_i)|+K\delta\bar{Y}, \ \forall x\in B_{\delta'}(x_0)\cap\Omega,
\end{equation}
which is a uniform bound for $\rho$ in $B_{\delta'}(x_0)\cap\Omega$. But then any curve asymptotic to $x_0 \in \partial \Omega$ would have finite length with respect to the conformal metric $\hat{g}=e^{2\rho}g_{\mathbb{S}^n}$, which contradicts its completeness and finishes the proof. 
\end{proof}

\section{Normal Flow, Embeddedness and the Bernstein Theorem}\label{Sect:FlowEmb}
An important issue in the theory of hypersurfaces is to know when an immersed hypersurface is in fact embedded. From the approach taken in \cite{bonini2010correspondences} where hypersurfaces were constructed via the normal flow of a subdomain of the infinity of hyperbolic space, it is naturally expected that immersed hypersurfaces are embedded when the normal flow is regular. It turns out that one can use convexity arguments to show that an admissible hypersurface can be unfolded along the normal flow when the structure of the boundary at infinity is relatively simple \cite{bonini2015hypersurfaces}. This makes the Global Correspondence Theorem \ref{Thm:GCT} more useful.

Before stating the main result we recall that the geodesic normal flow $\{\phi^t\}_{t \in \mathbb{R}}$ in $\mathbb{H}^{n+1}$ of an admissible hypersurface $\phi: \Omega \subset \mathbb{S}^n \to \mathbb{H}^{n+1}$ is given by
\begin{equation}\label{Eq:NormalFlow}
\phi^t(x) := \exp_{\phi(x)}(-t\eta(x)) = \phi(x) \cosh{t} - \eta(x)\sinh{t}:\Omega \to \mathbb{H}^{n+1} \subset \mathbb{R}^{1,n+1}
\end{equation}\label{Eq:Riccatti}
and the principal curvatures $\kappa_i^t$ of $\phi^t$ are given by
\begin{equation}\label{Eq:FlowPC}
\kappa_i^t = \frac{\kappa_i + \tanh{t}}{1+\kappa_i \tanh{t}}
\end{equation}
due to the Riccatti equations. Moreover, it is easily seen that the hyperbolic Gauss map $G^t$ is invariant under the normal flow and that the horospherical metric $g^t_h$ of $\phi^t$ is given by $g^t_h=e^{2t}g_h$ where $g_h$ is the horospherical metric of $\phi$.

With the two dimensional correspondence in hand, in particular, the consistency of boundaries, the following embeddedness theorem follows immmediately for all dimensions $n\geq2$ from Theorem 3.6 of \cite{bonini2015hypersurfaces}.

\begin{thm}\label{Thm:EmbFlow2}[Embeddedness Theorem]
For $n \geq 2$, suppose that $\phi : M^n\to\mathbb{H}^{n+1}$ is an admissible hypersurface. In addition, assume that the boundary at infinity $\partial_{\infty}\phi(M)$ is a disjoint, finite union of smooth compact embedded submanifolds with no boundary in $\mathbb{S}^{n}$. Then $\phi$ can be unfolded into an embedded hypersurface along its geodesic normal flow eventually.

Equivalently, suppose that $\hat{g}=e^{2\rho}g_{\mathbb{S}^{n}}$ is a realizable metric on a domain $\Omega \subset \mathbb{S}^{n}$.  In addition, assume that the boundary $\partial\Omega$ is a disjoint, finite union of smooth compact embedded submanifolds with no boundary in $\mathbb{S}^{n}$. Then the admissible hypersurfaces
\begin{equation}\label{Eq:NormalFlow2}
\phi^t=\frac{e^{\rho+t}}{2}(1+e^{-2(\rho+t)}(1+|\nabla\rho|^2))(1,x)+e^{-(\rho+t)}(0,-x+\nabla\rho) : \Omega\to\mathbb{H}^{n+1}
\end{equation}
are embedded for $t$ sufficiently large.
\end{thm}

Next note from \eqref{Eq:LCmetric} that the weak horospherical convexity of an immersed hypersurface $\phi:M^n \to \mathbb{H}^{n+1}$ guarantees the local injectivity of the hyperbolic Gauss map, which in turn is sufficient to start the geodesic normal flow. In particular, when $\partial_{\infty}\phi(M) = \{x\} \subset \mathbb{S}^n$ is a single point, this observation allows us to remove the injectivity of the hyperbolic Gauss map $G$ assumptions from one direction of the Global Correspondence Theorem \ref{Thm:GCT2} and the Embeddededness Theorem \ref{Thm:EmbFlow2}.

\begin{thm}\label{Thm:CorrSinglePtBd2}
For $n\geq 2$, suppose that $\phi : M^n\to\mathbb{H}^{n+1}$ is an immersed, complete, uniformly weakly horospherically convex hypersurface with the boundary at infinity $\partial_{\infty}\phi(M)=\{x\} \subset \mathbb{S}^n$ a single point. Then it induces a complete conformal metric $\hat{g}=e^{2\rho}g_{\mathbb{S}^{n}}$ on the image of the hyperbolic Gauss map $G(M)\subset\mathbb{S}^{n}$ with bounded $2$-tensor $P$ defined as in \eqref{Eq:2TensorP} where $\rho$ is the horospherical support function of $\phi$ and
\begin{equation}\label{Eq:Boundary}
\partial_{\infty}\phi(M)=\partial G(M).
\end{equation}
Moreover, $\phi$ can be unfolded into an embedded hypersurface along its geodesic normal flow eventually.
\end{thm}

\begin{proof}
As we mentioned previously in this section, we can still define the geodesic normal flow with just local injectivity of the hyperbolic Gauss map. Let $\phi^t$ denote the geodesic normal flow of $\phi=\phi^0$ for $t\geq0$. From the evolution of the principal curvatures along the normal flow given by \eqref{Eq:FlowPC}, we know that $\phi^t$ is uniformly convex when $t$ is sufficiently large. Then the injectivity of the hyperbolic Gauss map of $\phi^t$ follows from extensions of Corollary 4.4 and Proposition 1.3 in \cite{epstein1987asymptotic} to higher dimensions. Furthermore, we know $\phi$ is admissible due to the invariance of the hyperbolic Gauss map under the geodesic normal flow and therefore the result follows from the Global Correspondence Theorem \ref{Thm:GCT2} and the Embeddededness Theorem \ref{Thm:EmbFlow2}.
\end{proof}

Finally, we apply Theorem \ref{Thm:CorrSinglePtBd2} and the work on elliptic Weingarten problems in \cite{bonini2015hypersurfaces} to establish a new, stronger Bernstein type theorem that does not a priori assume embeddedness or any additional curvature conditions.

\begin{thm}\label{Thm:Bernstein2}
For $n\geq2$, suppose that $\phi: M^n\to\mathbb{H}^{n+1}$ is an immersed, complete, uniformly weakly horospherically convex hypersurface with constant mean curvature. Then it is a horosphere if its boundary at infinity is a single point in $\mathbb{S}^n$.
\end{thm}

\begin{proof}
By Theorem \ref{Thm:CorrSinglePtBd2}, for $t$ sufficiently large the hypersurface $\phi^t$ defined by \eqref{Eq:NormalFlow} is a properly embedded, uniformly weakly horospherically convex hypersurface with single point boundary at infinity. Moreover, from \eqref{Eq:FlowPC} it follows that $\phi^t$ satisfies an elliptic Weingarten equation so the theorem follows from the generalized Bernstein Theorem 4.4 in \cite{bonini2015hypersurfaces}.
\end{proof}


\nocite{*}


\bibliographystyle{amsplain}
\bibliography{reference}

\end{document}